\documentclass[11pt]{amsart}
\usepackage{amsfonts,amsmath,amssymb}

\def\rhu{\hbox{$\rightharpoonup$}}

\newtheorem{definition}{Definition}[section]
\newtheorem{lemma}[definition]{Lemma}

\newtheorem{proposition}[definition]{Proposition}

\newtheorem{theorem}[definition]{Theorem}

\newtheorem{de}[definition]{Definition}
\def\bea{\begin{eqnarray*}}
\def\eea{\end{eqnarray*}}

\newcommand{\End}{{\rm End}}

\def\a{$\mbox{\u{a}}$}

\def\S{$\mbox{\c{S}}$}

\def\bea{\begin{eqnarray*}}
\def\eea{\end{eqnarray*}}

\begin{document}

\title{The antipode of a dual quasi-Hopf algebra with nonzero integrals is bijective}
\dedicatory{Dedicated to Fred Van Oystaeyen for his sixtieth birthday}
\author{M. Beattie}\thanks{ The first author's research was
supported by an NSERC Discovery Grant.}
\address{Department of Mathematics and Computer Science, Mount Allison University,
Sackville, New Brunswick, Canada E4L 1E6 }
\email{mbeattie@mta.ca}
\author{M.C. Iovanov}\thanks{The second author was partially supported by the contract nr. 24/28.09.07 with UEFISCU ``Groups, quantum groups, corings and representation theory" of CNCIS, PN II (ID\_1002)}
\address{State University of New York, 244 Mathematics Building,
Buffalo, NY 14260-2900, USA, and University of Bucharest, Fac. Matematica \& Informatica,
Str. Academiei nr. 14,
Bucharest 010014,
Romania}
\email{yovanov@gmail.com}
\author{\c{S}. Raianu}
\address{Mathematics Department, California State University, Dominguez Hills,
1000 E Victoria St, Carson CA 90747, USA}
\email{sraianu@csudh.edu}
\begin{abstract} For $A$ a Hopf algebra of arbitrary dimension  over a field $K$, it
is well-known that if $A$ has nonzero integrals, or, in other
words, if the coalgebra $A$ is co-Frobenius, then the space of
integrals is one-dimensional and the antipode of $A$ is bijective.
Bulacu and Caenepeel recently showed that if  $H$ is a dual
quasi-Hopf algebra with nonzero integrals, then the space of
integrals is one-dimensional, and the antipode is injective. In this
short note we show that  the antipode is bijective.
\end{abstract}
\maketitle
\date{}

\section{Introduction }

The definition of quasi-Hopf algebras and the dual notion of dual
quasi-Hopf algebras is motivated by quantum physics and dates back
to work of Drinfel'd \cite{D}. The theory of integrals for
quasi-Hopf algebras was studied in \cite{pvo, hn, bc}. In \cite{bc},
Bulacu and Caenepeel showed that a dual quasi-Hopf algebra is
co-Frobenius as a coalgebra if and only if it has a nonzero
integral. In this case, the space of integrals is one-dimensional
and  the antipode is injective, so that for finite dimensional dual
quasi-Hopf algebras
the antipode is bijective. In this note, we
use the ideas from a new short proof of the bijectivity of the
antipode for Hopf algebras by the second author \cite{i} to show
that the antipode of a dual quasi-Hopf algebra with integrals is
bijective, thus extending the classical result of Radford \cite{R}
for Hopf algebras.

In this paper we prove

\begin{theorem}\label{TH}
Let $H$ be a co-Frobenius dual quasi-Hopf algebra, equivalently,
a dual quasi-Hopf algebra having nonzero integrals.
Then the antipode of $H$ is bijective.
\end{theorem}

\section{Preliminaries}

In this section we briefly review the definition of a dual
quasi-Hopf algebra over a field $K$.  We refer the reader to
\cite{abe, dnr, sw} for the basic definitions and properties of
coalgebras and their comodules and of Hopf algebras. For the
definition of dual quasi-Hopf algebra we follow \cite[Section
2.4]{M}.

\begin{de} A
  dual quasi-bialgebra $H$ over   $K$ is a coassociative
coalgebra $(H, \Delta, \varepsilon)$   together with   a unit
$u:K\rightarrow H$, $u(1)=1$, and a not necessarily associative
multiplication $M:H\otimes H\rightarrow H$. The maps $u$ and $M$ are
coalgebra maps. We write $ab$ for $M(a \otimes b)$.  As well, there
is an element $\varphi\in(H\otimes H\otimes H)^*$ called the {\it
reassociator}, which is invertible with respect to the convolution
algebra structure of $(H\otimes H\otimes H)^*$.  The following
relations must hold for all $h,g,f,e\in H$:
\begin{eqnarray}
h_1(g_1f_1)\varphi(h_2,g_2,f_2) & = & \varphi(h_1,g_1,f_1)(h_2g_2)f_2 \label{e1}\\
1h=h1=h \label{e2}\\
\varphi(h_1,g_1,f_1e_1)\varphi(h_2g_2,f_2,e_2) & = &
\varphi(g_1,f_1,e_1)\varphi(h_1,g_2f_2,e_2)\varphi(h_2,g_3,f_3) \label{e3}\\
\varphi(h,1,g) = \varepsilon(h)\varepsilon(g) \label{e4}
\end{eqnarray}
\end{de}
Here we use Sweedler's sigma notation with the summation symbol
omitted.

\begin{de} A dual quasi-bialgebra $H$ is called a dual quasi-Hopf algebra
if  there exists an antimorphism $S$ of the coalgebra $H$ and
elements $\alpha,\beta\in H^*$ such that for all $h\in H$:
\begin{eqnarray}
S(h_1)\alpha(h_2)h_3=\alpha(h)1, & & h_1\beta(h_2)S(h_3)=\beta(h)1 \label{e5}\\
\varphi(h_1\beta(h_2),S(h_3),\alpha(h_4)h_5) & = &
\varphi^{-1}(S(h_1),\alpha(h_2)h_3,\beta(h_4)S(h_5))
=\varepsilon(h). \label{e6}
\end{eqnarray}
\end{de}
Let $H$ be a dual quasi-Hopf algebra. As in the Hopf algebra case, a
left integral on $H$ is an element $T\in H^*$ such that
$h^*T=h^*(1)T$ for all $h^* \in H^*$; the space of left integrals is
denoted by $\int_l$ and by \cite[Proposition 4.7]{bc} has dimension
0 or 1. Right integrals are defined analogously with space of right
integrals denoted by $\int_r$. Suppose $0 \neq T \in \int_l$.   It
is easily seen that $\int_l$ is a two sided ideal of the algebra
$H^*$, and $KT\subseteq Rat(H^*)$ with right comultiplication given
by $T\mapsto T\otimes 1$. Since for co-Frobenius coalgebras
$Rat(H^*)=Rat({}_{H^*}H^*)=Rat(H^*_{H^*})$,   $KT$ must have  left
comultiplication $T\mapsto a\otimes T$. By coassociativity, $a$ is a
grouplike element, called the distinguished grouplike of $H$. Then,
for all $h^* \in H^*$,
\begin{eqnarray}
Th^* & = & h^*(a)T. \label{e7.5}
\end{eqnarray}

 From \cite[Proposition 4.2]{bc}, the function $\theta^*:\int_l\otimes H\rightarrow
Rat(_{H^*}H^*)$
\begin{eqnarray}
\theta^*(T\otimes h)= \sigma(S(h_5)\otimes \alpha(h_6)h_7)(S(h_4)
\rightharpoonup T ) \sigma^{-1}(S(h_3)\otimes\beta(S(h_2))S^2(h_1))
\label{e7}
\end{eqnarray}
is an isomorphism of right $H$-comodules, where $\sigma :H\otimes
H\rightarrow H^*$  is defined by $\sigma(h\otimes g)(f)=\varphi
(f,h,g)$, $\sigma^{-1} $ is the convolution inverse of $\sigma$,
and, as usual, $(h \rightharpoonup T)(g) = T(gh)$.

\section{Proof of the theorem}
Let $H$ be a quasi-Hopf algebra with $0 \neq T \in \int_l$. As in
\cite{i}, for each right $H$-comodule $(M,\rho)$, we denote by
${}^aM$ the left $H$-comodule structure on $M$ defined by
$m_{-1}^a\otimes m_{0}^a=aS(m_1)\otimes m_0$, where
$\rho(m)=m_0\otimes m_1$. Denote the induced right $H^*$-module
structure on ${}^aM$ by $m\cdot^a h^*=h^*(m_{-1}^a)m_{0}^a =
h^*(aS(m_1))m_0$. By
  \cite[Corollary 4.4]{bc} the antipode $S$ of
  $H$ is injective, and therefore
has a left inverse $S^l$. Then, for $\sigma$ as above, we have the
following analogue of \cite[Proposition 2.5]{i}:
\newpage
\begin{proposition}\label{surjmap}
The map $p:{^a{H}}\to Rat(H^*)$ defined by
\begin{eqnarray}
p(h) & = &
\sigma(S(S^l(h_3))\otimes\alpha(S^l(h_2))S^l(h_1))*(h_4\rhu T)
*\sigma^{-1}(h_5\beta(h_6)\otimes S(h_7)) \label{e8}
\end{eqnarray}
is a surjective morphism of left $H$-comodules.
\end{proposition}
\begin{proof}
Let $\Psi:=\sigma(S(S^l(h_3))\otimes\alpha(S^l(h_2))S^l(h_1))$. Then
for $c^*\in H^*$ and $g\in H$: \bea
(p(h)*c^*)(g)&=&p(h)(g_1)c^*(g_2)\\
(\ref{e8})&=&\Psi(g_1)T(g_2h_4)\sigma^{-1}(h_5\beta(h_6)\otimes S(h_7))(g_3)c^*(g_4)\\
&=&\Psi(g_1)T(g_2h_4)\varphi^{-1}(g_3,h_5\beta(h_6),S(h_7))c^*(g_4)\\
&=&\Psi(g_1)T(g_2h_4)\varphi^{-1}(g_3,h_5,S(h_7))c^*(g_4\beta(h_6))\\
(\ref{e5})&=&\Psi(g_1)T(g_2h_4)\varphi^{-1}(g_3,h_5,S(h_9))c^*(g_4(h_6\beta(h_7)S(h_8)))\\
&=&\Psi(g_1)T(g_2h_4)c^*(\varphi^{-1}(g_3,h_5,S(h_9))g_4(h_6\beta(h_7)S(h_8)))\\
&=&\Psi(g_1)T(g_2h_4)\beta(h_7)c^*(\varphi^{-1}(g_3,h_5,S(h_9))g_4(h_6S(h_8)))\\
(\ref{e1})&=&\Psi(g_1)T(g_2h_4)\beta(h_7)c^*((g_3h_5)S(h_9)\varphi^{-1}(g_4,h_6,S(h_8)))\\
&=&\Psi(g_1)T(g_2h_4)(S(h_9) \rightharpoonup c^* )(g_3h_5)\beta(h_7)\varphi^{-1}(g_4,h_6,S(h_8))\\
&=&\Psi(g_1)(T*(S(h_8)\rightharpoonup c^*))(g_2h_4)\beta(h_6)\varphi^{-1}(g_3,h_5,S(h_7))\\
(\ref{e7.5})&=&\Psi(g_1)(S(h_8) \rightharpoonup c^*)(a)T(g_2h_4)\beta(h_6)\varphi^{-1}(g_3,h_5,S(h_7))\\
&=&\Psi(g_1)T(g_2h_4)\sigma^{-1}(h_5\beta(h_6)\otimes S(h_7))(g_3)c^*(aS(h_8))\\
(\ref{e8})&=&p(h_1)(g)c^*(aS(h_2))\\
&=&p(c^*(aS(h_2))h_1)(g)\\
&=&p(c^*(h^a_{-1})h^a_{0})(g)\\
&=&p(h\cdot^a c^*)(g). \eea Thus $p$ is left $H$-colinear. Finally,
  we note that $p\circ S = \theta^*(T\otimes -)$  where $\theta^*$
  is the isomorphism from (\ref{e7}) so that $p$ is surjective.
\end{proof}
\vspace{2mm}
 Let $c $ be a grouplike element of $H$. From \cite[p.580]{bc}, $c$ is invertible
 with inverse $S(c)$.    We will show that left multiplication
by $c$ has an inverse too.
 \par Let $\theta_c\in\End(H)$ be defined by $\theta_c(h)=ch$ and define the coinner automorphisms
 $q_c$ and $r_c= q_c^{-1} \in\End(H)$ by:
 $$q_c(h)=
 \varphi^{-1}(c,S(c),h_1)h_2\varphi(c,S(c),h_3) \mbox{ and }
  r_c(h)=
 \varphi(c,S(c),h_1)h_2\varphi^{-1}(c,S(c),h_3).$$
\vspace{1mm}
\begin{lemma}\label{theta}
   For any grouplike element $c$ and $\theta_c, r_c,q_c$ as
above,
  $\theta_c\circ\theta_{c^{-1}}=r_c$ and thus $\theta_c$ is bijective
with inverse $\theta_c^{-1}=\theta_{c^{-1}}\circ
q_c=q_{c^{-1}}\circ\theta_{c^{-1}}$.
\end{lemma}
\begin{proof}

  Using (\ref{e1}) and the fact that $c^{-1}=S(c)$,
we see that

$$\theta_c\circ\theta_{c^{-1}}(h) =
c(c^{-1}h)= \varphi(c,S(c),h_1)(cS(c))h_2\varphi^{-1}(c,S(c),h_3)
   =
 r_c(h).$$

The same formula for $c^{-1}=S(c)$ yields
$\theta_{c^{-1}}\circ\theta_c=r_{c^{-1}}$ and the statement then
follows directly.
\end{proof}

\newpage
We can now prove our main result.\\[.5cm]
{\it Proof of Theorem \ref{TH}.}\\
We only need to prove the surjectivity. The proof goes along the
lines of the proof of \cite[Theorem 2.6]{i}, but with the difference
that here the antipode is not necessarily an anti-morphism of
algebras.
\par Let $\pi$ be the composition map
${}^aH\stackrel{p}{\rightarrow}Rat(H^*_{H^*})\stackrel{\sim}{\rightarrow}H\otimes
\int_r\simeq H$, where the last two isomorphisms follow by
left-right symmetry of the results of \cite{bc}. Since $H$ is a
co-Frobenius coalgebra, ${^H\!{H}}$ is projective by \cite[Theorem
1.3]{GTN} or \cite[Theorem 4.5, (x)]{bc}, and as $\pi$ is
surjective, there is a morphism of left $H$-comodules
 $\lambda:H\rightarrow {}^aH$
such that $\pi\lambda={\rm Id}_H$. We then have
$$ aS(\lambda(h)_2)\otimes\lambda(h)_1 =
  \lambda(h)^a_{-1}\otimes \lambda(h)^a_{0}=h_1\otimes\lambda(h_2).
$$
Applying ${\rm Id}\otimes \varepsilon\pi$, we get
$aS(\varepsilon\pi(\lambda(h)_1)\lambda(h)_2)=h$ for any $h \in H$.
Thus $\theta_a \circ S$ is surjective and since $\theta_a$ is
bijective by Lemma \ref{theta}, $S$ is surjective also.
$\hfill\square$


\end{document}